\documentclass[12pt]{amsart}
\usepackage[latin1]{inputenc}
\usepackage[english]{babel}
\usepackage{amsmath}
\usepackage{amsthm}
\usepackage{amssymb}
\usepackage{amsmath, amsthm, amsfonts, mathrsfs, amsfonts}
\usepackage{latexsym}
\usepackage{textcomp}
\usepackage{enumerate}
\usepackage{ulem}

\usepackage{xcolor}

\theoremstyle{definition}
\newtheorem{definition}{Definition}[section]
\newtheorem*{definition1}{Definition}
\newtheorem{question}{Question}[]

\theoremstyle{remark}

\theoremstyle{plain}

\theoremstyle{plain}

\theoremstyle{plain}

\theoremstyle{plain}
\newtheorem{lemma}[definition]{Lemma}

\theoremstyle{plain}

\newtheorem{prop}[definition]{Proposition}
\newtheorem{theoremA}[]{Theorem}

\newtheorem{corA}[theoremA]{Corollary}

\theoremstyle{remark}

\theoremstyle{remark}

\theoremstyle{definition}

\newcommand{\A}{\mathbf{A}}

\newcommand{\N}{\mathrm{N}}

\newcommand{\Syl}{\mathrm{Syl}}

\newcommand{\PSL}{\mathrm{PSL}}

\newcommand{\norm}{\mathrel{\unlhd}}

\def \Syl {\hbox {\rm Syl}}

\def \ov {\overline}

\title{$p$-nilpotency criteria for some verbal subgroups}

\author[Y. Contreras Rojas]{Yerko Contreras Rojas}\address{
Faculty of Mathematics -- Institute of exact sciences \\
Universidade Federal do Sul e Sudeste do Par\'a \\
Avenida dos Ip\^es, Cidade Universit\'aria, Marab\'a - Par\'a  \\ Brazil} \email{yerkocr@unifesspa.edu.br}

 \author[V. Grazian]{Valentina Grazian}\address{
Department of Mathematics and Applications\\
University of Milano -- Bicocca\\
Via Roberto Cozzi 55, 20125 Milano \\ Italy} \email{valentina.grazian@unimib.it}

\author[C. Monetta]{Carmine Monetta}\address{Department of Mathematics \\ University of Salerno \\
via Giovanni Paolo II 132, 84084 Fisciano (SA)\\ Italy}
\email{cmonetta@unisa.it}


\begin{document}
\maketitle

\begin{abstract}
    Let $G$ be a finite group, let $p$ be a prime and let $w$ be a group-word.
    We say that $G$ satisfies $P(w,p)$ if the prime $p$ divides the order of $xy$ for every $w$-value $x$ in $G$ of $p'$-order and for every non-trivial $w$-value $y$ in $G$ of order divisible by $p$.
    With $k \geq 2$, we prove that the $k$th term of the lower central series of $G$ is $p$-nilpotent if and only if $G$ satisfies $P(\gamma_k,p)$.
    In addition, if $G$ is soluble, we show that the $k$th term of the derived series of $G$ is $p$-nilpotent if and only if $G$ satisfies $P(\delta_k,p)$.
\end{abstract}

\section{Introduction}

Let $G$ be a finite group and  let $p$ be a prime. We say that $G$ is $p$-nilpotent if $G$ has a normal $p$-complement, that is, a normal subgroup $H$ of $G$ having $p'$-order and $p$-power index in $G$.
There are well known results giving sufficient conditions  for the  $p$-nilpotecy of a finite group, such as Burniside's theorem \cite[Theorem 7.4.3]{GOR}, and remarkable criteria like Frobenius' normal $p$-complement theorem  \cite[Theorem 7.4.5]{GOR}, which says that a group $G$ is $p$-nilpotent if and only if $N_G(H)/C_G(H)$ is a $p$-group for every $p$-subgroup $H$ of $G$. These results mainly deal with properties of $p$-local subgroups and they motivated other works in the subject (see for instance \cite{Asaad} and \cite{BX}).
In recent years, the problems of nilpotency and $p$-nilpotency of a group have been studied from a different point of view.
In an unpublished work of 2014 \cite{BW}, Baumslag and Wiegold studied the nilpotency of a finite group looking at the behaviour of elements of coprime orders. 
Beltr\'an and S\'aez \cite{BS} took inspiration to characterize the $p$-nilpotency of a finite group by means of the orders of its elements:

\begin{theoremA}\label{BS}\cite[Theorem A]{BS}
Let $G$ be a finite group and let $p$ be a prime. Then $G$ is $p$-nilpotent if and only if for every $p'$-element $x$ of $G$ of prime power order and for every $p$-element $y\neq 1$ of $G$, $p$ divides $o(xy)$.
\end{theoremA}

Here $o(x)$ denotes the order of the group element $x$. 
We point out that, instead of studying the couple of elements $(x,y)$, where $x$ is a $p'$-element of prime power order and $y$ is a non-trivial $p$-element, one can focus on the couple $(x,y)$, where $x$ is a $p'$-element and $y$ is a non-trivial element of order divisible by $p$.

\begin{corA}\label{thm.intro}
Let $G$ be a finite group and let $p$ be a prime. Then $G$ is $p$-nilpotent if and only if  for every $x\in G$ such that $p$ does not divide $o(x)$ and for every $1 \neq y\in G$ such that $p$ divides $o(y)$, $p$ divides $o(xy)$.
\end{corA}

The work of Baumslag and Wiegold has been extended to the realm of group-words in a series of papers \cite{BM,Carmine.meta, BS1,Pavel,comm.p'}.
 In a similar matter, the aim of this work is to generalize Theorem \ref{BS} obtaining a verbal version.

A group-word is any nontrivial element of a free group $F$ on free generators $x_1,x_2,\dots$, that is, a product of finitely many $x_i$'s and their inverses. The elements of the commutator subgroup of $F$ are called commutator words. Let $w=w(x_1,\dots,x_k)$ be a group-word in the variables $x_1,\dots, x_k$. For any group $G$  and arbitrary $g_1,\dots, g_k\in G$, the elements of the form $w(g_1,\dots, g_k)$ are called $w$-values in $G$. We denote by $G_w$ the set of all $w$-values in $G$. The verbal subgroup of $G$ corresponding to $w$ is the subgroup $w(G)$ of $G$ generated by $G_w$.

\begin{definition1} Let $G$ be a group, let $p$ be a prime and let $w$ be a group-word. We say that $G$ satisfies $P(w,p)$ if the prime $p$ divides $o(xy)$, for every $x\in G_w$ of $p'$-order and for every non-trivial $y\in G_w$ of order divisible by $p$.
\end{definition1}

Let $w$ be a group-word and let $p$ be a prime. If $G$ is a finite group and $w(G)$ is $p$-nilpotent, then Corollary~\ref{thm.intro} implies that $G$ satisfies $P(w,p)$. An interesting question is the following:

\begin{question}\label{question}
If $G$ is a finite group satisfying $P(w,p)$, is  $w(G)$ $p$-nilpotent?
\end{question} 

In general the answer is negative. For instance, one may consider any non-abelian simple group $G$, say of exponent $e$, and the word $x^n$, where
$n$ is a divisor of $e$ such that $e/n$ is prime. If $p$ is a prime dividing the order of $G$, then $G$ satisfies $P(w,p)$, but $w(G) = G$ is not $p$-nilpotent.

Even in the case of commutator words we can find counterexamples. Indeed, if $G= Alt(5)$ is the alternating group of degree $5$ and $w$ is the word considered in \cite[Example 4.2]{comm.p'}, then $G_w$ consists of the identity and all products of two transpositions. In particular if $p\in \{2,3,5\}$ then $G$ satisfies $P(w,p)$, but $w(G) = G$ is a simple group and therefore not $p$-nilpotent.

However, if we consider the group-word $w=x$, Corollary~\ref{thm.intro} says exactly that $w(G)$ is $p$-nilpotent if and only if G satisfies $P(w,p)$. Actually, this situation is a particular instance of a more general result concerning the commutator word $\gamma_k$.

Given an integer $k\geq 1$, the word $\gamma_k = \gamma_k(x_1, \dots, x_k)$ is defined inductively by the formulae
\[ \gamma_1 = x_1, \quad \gamma_k = [\gamma_{k-1},x_k] = [x_1, \dots, x_k] ~ \text{ for } k\geq 2,\]
where $[x,y] = x^{-1}y^{-1}xy$, for any group elements $x$ and $y$.
Note that $\gamma_2(G) = G'$ is the derived subgroup of $G$ and in general the subgroup $\gamma_k(G)$ of $G$ corresponds to the $k$-th term of the lower central series of $G$.

Our first main theorem is the following:

\begin{theoremA}\label{thm.gamma}
Let $G$ be a finite group, let $k\geq 1$ be an integer and let $p$ be a prime. Then $\gamma_k(G)$ is $p$-nilpotent if and only if $G$ satisfies $P(\gamma_k,p)$.
\end{theoremA}

It is worth mentioning that Theorem \ref{thm.gamma} gives a positive answer to Question \ref{question} for $\gamma_k$-words. Such words belong to the larger class of multilinear commutator words, which are words obtained by nesting commutators but using always different variables. For example the word $[[x_1,x_2],[x_3,x_4,x_5],x_6]$ is a multilinear commutator, while the Engel word $[x_1,x_2,x_2]$ is not. Hence, it is natural to ask if the answer to Question \ref{question} remains positive if $G$ is any finite group and the word considered is any multilinear commutator word. In this direction, we provide another positive answer when $w$ belongs to the family of $\delta_k$-words. Given an integer $k\geq 0$, the word $\delta_k = \delta_k(x_1, \dots, x_{2^k})$ is defined inductively by the formulae
\[ \delta_0 = x_1, \quad \delta_k = [\delta_{k-1}(x_1, \dots, x_{2^{k-1}}),\delta_{k-1}(x_{2^{k-1}+1}, \dots, x_{2^k})] ~ \text{ for } k\geq 1.\]
We have $\delta_1(G)= G'$ and $\delta_k(G)$ corresponds to the $k$-th term  $G^{(k)}$ of the derived series of $G$.
Since $\delta_h = \gamma_{h+1}$ for $0 \leq h \leq 1$, Theorem \ref{thm.gamma} gives an affirmative answer to Question~\ref{question} when $w=\delta_0, \delta_1$. 
For $w=\delta_k$ with $k \geq 2$, we prove the following: 

\begin{theoremA}\label{thm.delta.sol}
Let $G$ be a finite soluble group, let $k\geq 2$ be an integer and let $p$ be a prime. Then $G^{(k)}$ is $p$-nilpotent if and only if $G$ satisfies $P(\delta_k,p)$.
\end{theoremA}

\section{Preliminaries}
Let $G$ be a finite group and let $p$ be a prime. We denote by $O_p(G)$ the largest normal $p$-subgroup of $G$ and by $O_{p'}(G)$ the largest normal $p'$-subgroup of $G$.
If $G$ is $p$-nilpotent, then the normal $p$-complement of $G$ is unique and corresponds to the group $O_{p'}(G)$. In particular we can write $G=PO_{p'}(G)$ for some Sylow $p$-subgroup $P$ of $G$. Moreover, if $G$ is $p$-nilpotent then $O_{p'}(G)$ contains all elements of $G$ having $p'$-order. Now, Corollary \ref{thm.intro} is a direct consequence of Theorem \ref{BS} and basic properties of $p$-nilpotent groups.

\begin{proof}[Proof of Corollary \ref{thm.intro}]
 Suppose $G$ is $p$-nilpotent. Then it has a normal $p$-complement $H$. Let $x,y \in G$ be such that $p$ does not divide  $o(x)$ and $p$ divides $o(y)$, with $y \neq 1$. Then $x \in H$ and $xy \not\in H$ as $y\not\in H$. Hence $p$ divides $o(xy)$.
 The other implication follows immediately from Theorem \ref{BS}.
\end{proof}

If $G$ is a group, we denote by $F(G)$ the Fitting subgroup of $G$, that is the largest normal nilpotent subgroup of $G$, and by $Fit_p(G)$ the $p$-Fitting subgroup of $G$, that is the largest normal $p$-nilpotent subgroup of $G$. Note that $F(G) \leq Fit_p(G)$.

\begin{lemma}\label{fitp.eq.Op}
Let $G$ be a finite group and let $p$ be a prime. If $O_{p'}(G) = 1$ then $Fit_p(G) = O_p(G)$ is a $p$-group. 
\end{lemma}

\begin{proof}
We have $O_{p'}(Fit_p(G)) \leq O_{p'}(G) = 1$. Since $Fit_p(G)$ is $p$-nilpotent, we also have $Fit_p(G)=TO_{p'}(Fit_p(G))$ for some Sylow $p$-subgroup $T$ of $Fit_p(G)$. Hence $Fit_p(G)=T$ is a normal $p$-subgroup of $G$. Thus $Fit_p(G) \leq O_p(G)$. On the other hand, $O_p(G)$ is a normal $p$-nilpotent subgroup of $G$, and so we conclude that $Fit_p(G) = O_p(G)$.
\end{proof}

A group $G$ is said to be metanilpotent if it has a normal subgroup $N$ such that both $N$ and $G/N$ are nilpotent.

\begin{lemma}\label{meta}\cite[Lemma 3]{Carmine.meta} Let $p$ be a prime and $G$ a finite metanilpotent group.  Suppose that $x$ is a $p$-element in $G$ such that $[O_{p'}(F(G)),x]=1$. Then $x\in F(G)$.
\end{lemma}

We recall that a subset $X$ of a group $G$ is said to be commutator-closed if $[X,X] \subseteq X$, and symmetric if $X^{-1}=X$. 
In \cite[Lemma 2.1]{Pavel}, it has been showed that a finite soluble group $G$ admits a commutator-closed subset $X$ such that $G=\langle X \rangle$ and every element of $X$ has prime power order. From the proof, it is easy to check that such an $X$ is also symmetric. Hence we have the following. 
\begin{lemma}\label{Pavel.sol}
Let $G$ be a finite group. If $G$ is soluble then there exists a commutator-closed and symmetric subset $X$ of $G$ such that $G=\langle X \rangle$ and every element of $X$ has prime power order.
\end{lemma}

Next results give sufficient conditions for a group to be generated by certain $w$-values of prime power order.

\begin{lemma}\label{Pavel.general}
Let $G$ be a finite soluble group. Then for every $k\geq 2$ the group $\gamma_k(G)$ is generated by $\gamma_k$-values of prime power order.
\end{lemma}

\begin{proof}
By Lemma \ref{Pavel.sol} there exists a commutator-closed and symmetric subset $X$ of $G$ such that $G=\langle X \rangle$ and every element of $X$ has prime power order. Hence by  \cite[Lemma 3.6 item (c)]{KHU1}
we get $\gamma_k(G) = \langle [x_1, \dots, x_t] \mid x_i \in X \cup X^{-1}, t\geq k\rangle$. Note that $X \cup X^{-1}=X$, and $[x_1, \dots, x_t]$ is a $\gamma_k$-value for every $t\geq k$. Also, since $X$ is commutator-closed, we get that  $[x_1, \dots, x_t] \in X$ has prime power order. Thus, $\gamma_k(G)$ is generated by $\gamma_k$-values of prime power order.
\end{proof}

\begin{lemma}\label{lem:gengamma}\cite[Lemma 4]{Carmine.meta}
Let $k$ be a positive integer and $G$ a finite group such that $G=G'$. If
$p$ is a prime dividing the order of $G$, then $G$ is generated by $\gamma_k$-values of $q$-power order for primes $q \neq p$.
\end{lemma}

\begin{lemma}\label{delta.focal}\cite[Lemma 2.5]{Pavel}
Let $G$ be a finite soluble group and let $Q$ be a Sylow $q$-subgroup of $G$. Then  for every $i \geq 1$ the group $Q \cap G^{(i)}$ can be generated by $\delta_i$-values lying in $Q$. 
\end{lemma}

 The following relation between the subgroups $\gamma_k(G)$ and $G^{(k)}$ of $G$ is an immediate consequence of \cite[Lemma 3.25]{KHU1}.

\begin{lemma}\label{delta.gamma.inclusion} If $G$ is a group, then for every $k\geq 0$ we have
$G^{(k)} \leq \gamma_{k+1}(G)$.
\end{lemma}

We conclude this section with a direct application of \cite[Theorem 1]{quasisimple}, which says that in almost every finite quasisimple group all elements are commutators.

\begin{prop}\label{thm.quasi} Let $G$ be a finite quasisimple group and suppose that $Z(G)$ is a $p$-group. Then every element of $G$ is a commutator, with the following exceptions:
\begin{enumerate}
    \item $p=3$, $G/Z(G) \cong \A_6$, $Z(G) \cong C_3$ and the non-central elements of $G$ that are not commutators have order $12$;
    \item $p=2$, $G/Z(G) \cong \PSL(3,4)$, $Z(G) \cong C_2 \times C_4$ and the non-central elements of $G$ that are not commutators have order $6$;
    \item $p=2$, $G/Z(G) \cong \PSL(3,4)$, $Z(G) \cong C_4 \times C_4$ and the non-central elements of $G$ that are not commutators have order $12$.
\end{enumerate}
\end{prop}

\section{Groups with Property $P(w,p)$}
In this section $w$ is any group-word. We study the properties of groups satisfying property $P(w,p)$ but being minimal such that $w(G)$ is not $p$-nilpotent.

\begin{definition}

We say that a group $G$ is a { \it minimal $P(w,p)$-exception} if 
\begin{itemize}
    \item $G$ satisfies $P(w,p)$,
\item $w(G)$ is not $p$-nilpotent,  
\item whenever $H$ is a proper subgroup of $G$, the group $w(H)$ is $p$-nilpotent;
\item whenever $G/N$ is a proper quotient of $G$ satisfying $P(w,p)$, the group $w(G/N)$ is $p$-nilpotent.
\end{itemize}
\end{definition}

Next lemma shows that  property $P(w,p)$ is closed with respect to forming subgroups and certain images.

\begin{lemma}\label{sbg.and.quotient.general}
Let $G$ be a finite group satisfying $P(w,p)$.
\begin{itemize}
    \item If $H \leq G$ is a subgroup of $G$ then $H$ satisfies $P(w,p)$.
    \item If $N\norm G$ is a normal subgroup of $G$ of $p'$-order, then $G/N$ satisfies $P(w,p)$.
\end{itemize}
\end{lemma}

\begin{proof}
If $H \leq G$ then for every $h\in H_w$ we have $h\in G_w$. Thus $H$ satisfies $P(w,p)$ if $G$ does.

Assume that $N$ is a normal subgroup of $G$ of $p'$-order and consider $\ov{G} = G/N$. Let $xN \in \ov{G}_w$ be a $w$-value of $p'$-order and let $1 \neq yN \in \ov{G}_w$ be a $w$-value of order divisible by $p$. Then we can assume that $x, y \in G_w$. Also, $p$ divides the order of $y$ and, since $N$ has $p'$-order, $p$ does not divide the order of $x$. Thus by  $P(w,p)$ we deduce that $p$ divides the order of $xy$.
In particular $xy \notin N$ and $p$ divides the order of $xyN = xN \cdot yN$. This shows that $\ov{G}$ satisfies property $P(w,p)$. 
\end{proof}

\begin{lemma}\label{Op'.word}
If $G$ is a  minimal $P(w,p)$-exception then $O_{p'}(G) = 1$ and $Fit_p(G) = O_p(G)$.
\end{lemma}

\begin{proof}
 Set $N=O_{p'}(G)$.
Aiming for a contradiction, suppose $N \neq 1$ and consider $\ov{G} = G/N$.
Then $|\ov{G}| < |G|$ and by Lemma \ref{sbg.and.quotient.general} the group $\ov{G}$ satisfies $P(w,p)$. Since $G$ is a  minimal $P(w,p)$-exception we deduce that $w(\ov{G}) = w(G)N/N$ is $p$-nilpotent.  Then $w(G)N/N = S/N \cdot H/N$ where $S/N \in \Syl_p(w(G)N/N)$ and  $H/N = O_{p'}(w(G)N/N)$. 
Set $K= H \cap w(G)$. Since $H \norm w(G)N$, we deduce that $K \norm w(G)$. Also, $|K|$ divides $|H| = [H \colon N]|N|$ and so $|K|$ is prime to $p$. Finally, $[w(G) \colon K] = [w(G)H \colon H]$ divides $[w(G)N \colon H]$, that is a power of $p$. We deduce that $K$ is a normal $p$-complement of $w(G)$ and $w(G)$ is $p$-nilpotent, a contradiction. Therefore we must have $O_{p'}(G) = 1$. The second part of the statement now follows from Lemma \ref{fitp.eq.Op}.

\end{proof}

\begin{lemma}\label{normal.sylow} Suppose $G$ is a minimal $P(w,p)$-exception, $G > G'$ and $G'/w(G')$ is nilpotent. Then $G'$ has a normal Sylow $p$-subgroup.
\end{lemma}

\begin{proof}
Note that $G'$ satisfies $P(w,p)$ by Lemma \ref{sbg.and.quotient.general}.
Since $G$ is a  minimal $P(w,p)$-exception and $|G'| < |G|$ we deduce that $w(G')$ is $p$-nilpotent. Note that $w(G')$ is a normal subgroup of $G'$, so $O_{p'}(w(G')) \leq O_{p'}(G') \leq O_{p'}(G)=1$ by Lemma \ref{Op'.word}. Hence $O_{p'}(w(G')) =1$ and since $w(G')$ is $p$-nilpotent we deduce that $w(G')$ is a $p$-group. In particular $w(G')$ is nilpotent. By assumption $G'/w(G')$ is nilpotent. Hence the group $G'$ is metanilpotent. Let $F(G')$ denote the Fitting subgroup of $G'$. Then $F(G') \leq Fit_p(G') \leq Fit_p(G) = O_p(G)$ (again by Lemma \ref{Op'.word}), so $F(G')$ is a $p$-group. Let $P\in \Syl_p(G')$ be a Sylow $p$-subgroup. Then $F(G')\leq P$. On the other hand, $[P, O_{p'}(F(G'))] =[P, 1] = 1$ and by Lemma \ref{meta} we get $P \leq F(G')$.
Therefore $P= F(G')$ is normal in $G'$.
\end{proof}

\begin{lemma}\label{self.centered.sylow} Suppose $G$ is a minimal $P(w,p)$-exception and soluble. If $H$ is a normal subgroup of $G$ and $P\in \Syl_p(H)$ then $C_H(P) \leq P$.
\end{lemma}

\begin{proof}
By Lemma \ref{Op'.word} we have $O_{p'}(G) = 1$. Since $H$ is normal in $G$ we have $O_{p'}(H) = 1$ and so $Fit_p(H)=O_p(H) \leq P$ (with the same proof used to show that $Fit_p(G)=O_p(G)$). Then $F(H) \leq Fit_p(H) \leq P$ and so $C_{H}(P) \leq C_{H}(F(H))$. Since $G$ is soluble, $H$ is soluble as well and by  \cite[Theorem 6.1.3]{GOR} we get 
\[ C_{H}(P) \leq C_{H}(F(H)) \leq F(H) \leq Fit_p(H) \leq P.\]
Therefore $C_H(P) \leq P$.
\end{proof}

\section{Proof of Theorem \ref{thm.gamma}}
In this section we consider the word $w=\gamma_k$ and we prove Theorem \ref{thm.gamma}.

Note that if a finite group $G$ satisfies property $P(\gamma_2,p)$ then it satisfies property $P(\gamma_k,p)$ for every $k\geq 2$. However, in general the converse is not true. As an example, consider the group $G$ generated by $g_1, g_2, g_3, h_1, h_2$ subject to the following relations
\[g_i^2 = h_j^3 = 1, \ \  [g_2,g_1] = g_3, \ \ [h_1,g_1]= [h_1,g_3] = h_1, \ \ [h_2, g_3] = h_2, h_1^{g_2} = h_2, \ \  h_2^{g_2} = h_1.
\]

Then $G$ is a group of order $72$ with $\gamma_3(G) \cong C_3 \times C_3$, thus $G$ satisfies property $P(\gamma_3,3)$. However $h_2 = [h_2,g_3]$ has order $3$, $g_3 = [g_2,g_1]$ has order 2 and $h_2g_3$ has order $2$. Thus $G$ does not satisfy property $P(\gamma_2,3)$.

\begin{lemma}\label{G2.G3}
Let $G$ be a quasisimple group. Then $G_{\gamma_2} = G_{\gamma_k}$ for every $k\geq 2$. In particular if $G$ is quasisimple, $p$ is a prime and $k \geq 2$, then $G$ satisfies property $P(\gamma_k,p)$ if and only if it satisfies property $P(\gamma_2,p)$.
\end{lemma}

\begin{proof}
It is enough to prove that $G_{\gamma_2} = G_{\gamma_3}$, then the result will follow by induction on $k$.
Clearly every $\gamma_3$-value is a commutator, so $G_{\gamma_3} \subset G_{\gamma_2}$.

Now, let $g=[a,b]\in  G_{\gamma_2}$ be a commutator. Using Ore's conjecture and the fact that $G$ is quasisimple, we can write $a=xz$ where $x$ is a commutator and $z\in Z(G)$. So $g=[xz,b] = [x,b] \in G_{\gamma_3}$. Therefore $G_{\gamma_2} \subset G_{\gamma_3}$. This completes the proof.
\end{proof}

\begin{lemma}\label{p'.word} Let $G$ be a finite group, $p$ a prime and $k \geq 2$. Suppose that $G$ satisfies $P(\gamma_k,p)$ and let $x\in G_{\gamma_k}$ be a $\gamma_k$-value of $p'$-order. Then for every element $g \in G$, the element $[g,_{k-1} x]$ has $p'$-order.
\end{lemma}

\begin{proof}
Note that for every $g \in G$ we have
\[ [g,_{k-1} x] \cdot x^{-1} = x^{-[g,_{k-2}x]}.\]
Aiming for a contradiction, suppose there exists $g\in G$ such that $p$ divides the order of $[g,_{k-1} x]$. Then $[g,_{k-1} x] \neq 1$ and by $P(\gamma_k,p)$ we deduce that $p$ divides the order of  $[g,_{k-1} x] \cdot x^{-1}$. Thus $p$ divides the  order of $x^{-[g,_{k-2}x]}$, that coincides with the order of $x$, a contradiction. This proves the statement. 
\end{proof}

\begin{lemma}\label{p.sbg.word} Let $G$ be a finite group satisfying $P(\gamma_k,p)$. If $P$ is a $p$-subgroup of $G$ and $x\in N_G(P)$ is such that $x \in G_{\gamma_k}$ has $p'$-order, then $[P,x] =1$.
\end{lemma}

\begin{proof}

Let $g\in P$ be an element. By Lemma \ref{p'.word} the element $[g,_{k-1} x] $ has $p'$-order. 
On the other hand, since $x \in N_G(P)$ we have that $[g,_{k-1} x] \in P$. Therefore the only possibility is $[g,_{k-1} x]= 1$. This shows that $[P,_{k-1}x] = 1$. Therefore by \cite[Theorem 5.3.6]{GOR} we get $[P,x] = [P,_{k-1}x] = 1$.
\end{proof}

\begin{lemma}\label{gamma.quasisimple}
Let $G$ be a finite group and $p$ a prime. Suppose $G=G'$, $G$ satisfies $P(\gamma_k,p)$ and assume that $G$ is minimal (with respect to the order) such that $\gamma_k(G)$ is not $p$-nilpotent. Then $Fit_p(G) = O_p(G)= Z(G)$ and $G$ is quasisimple.
\end{lemma}

\begin{proof}
Note that $G$ is a minimal $P(\gamma_k,p)$-exception. Hence by Lemma \ref{Op'.word} we have $Fit_p(G)=O_p(G)$.
By Lemma~\ref{lem:gengamma} the group $G$ is generated by $\gamma_k$-values of $p'$-order. Since $Fit_p(G)=O_p(G)$, Lemma \ref{p'.word} implies that $Fit_p(G) \leq Z(G)$. On the other hand, $Z(G)$ is abelian and so $p$-nilpotent and it is normal in $G$, so $Z(G) \leq Fit_p(G)$. Thus $Fit_p(G)=Z(G)$.

It remains to prove that $G$ is quasisimple. Since $G=G'$, it is enough to show that $G/Z(G)$ is simple.
Let $N/Z(G) \norm G/Z(G)$ be a proper normal subgroup of $G/Z(G)$. Then $N \norm G$ is a proper subgroup and $N$ has property $P(\gamma_k,p)$ by Lemma \ref{sbg.and.quotient.general}. Hence $\gamma_k(N)$ is $p$-nilpotent by minimality of $G$. Therefore $\gamma_k(N) \leq Fit_p(G) = Z(G)$. In particular $\gamma_{k+1}(N)= [\gamma_k(N), N] = 1$ and so $N$ is nilpotent and $N \leq Fit_p(G) = Z(G)$. Thus $N/Z(G) = 1$. This shows that $G/Z(G)$ is simple and completes the proof.   
\end{proof}

\begin{lemma}\label{no.exceptions} 
Let $G$ be a finite quasisimple group, $p$ a prime and $k \geq 2$. If $G$ satisfies $P(\gamma_k,p)$ and $Z(G)$ is a $p$-group, then every element of $G$ is a $\gamma_k$-value.
\end{lemma}

\begin{proof}
If every element of $G$ is a commutator, then by induction on $k$ we deduce that every element of $G$ is a $\gamma_k$-value.
Suppose $G$ contains elements that are not commutators. Then by Proposition \ref{thm.quasi} one of the following holds:
\begin{enumerate}
    \item $p=3$, $G/Z(G) \cong \A_6$ and $Z(G) \cong C_3$;
    \item $p=2$, $G/Z(G) \cong \PSL(3,4)$ and $Z(G) \cong C_2 \times C_4$;
    \item $p=2$, $G/Z(G) \cong \PSL(3,4)$ and $Z(G) \cong C_4 \times C_4$.
\end{enumerate}
Using GAP we can check that none of the groups in the above list satisfies property $P(\gamma_2,p)$ (where $p=3$ in the first case and $p=2$ in the others). Hence by Lemma \ref{G2.G3} the groups in the list do not satisfy property $P(\gamma_k,p)$, a contradiction.
\end{proof}

\begin{lemma}\label{gamma.G soluble}
Let $G$ be a finite group, $p$ a prime and $k \geq 2$. Suppose $G$ satisfies $P(\gamma_k,p)$ and is minimal (with respect to the order) such that $\gamma_k(G)$ is not $p$-nilpotent. Then $G$ is soluble and  $G'$ has a normal Sylow $p$-subgroup.
\end{lemma}

\begin{proof}
We first show that $G > G'$. Aiming for a contradiction, suppose $G=G'$. Then by Lemma \ref{gamma.quasisimple} we deduce that $G$ is quasisimple and $Z(G)=O_p(G)$ is a $p$-group. By Lemma \ref{no.exceptions} every element of $G$ is a $\gamma_k$-value. Let $P\leq G$ be a $p$-subgroup and let $y\in \N_{G}(P)$ be an element of $p'$-order. Then $y$ is a $\gamma_k$-value of $p'$-order and by Lemma \ref{p.sbg.word} we deduce that $[P,y] = 1$. Hence $N_G(P)/C_G(P)$ is a $p$-group. By Frobenius criterion we conclude that $G$ is $p$-nilpotent, a contradiction. Hence $G>G'$.
The fact that $G'$ has a normal Sylow $p$-subgroup follows from Lemma \ref{normal.sylow}.

It remains to show that $G$ is soluble. Since $G$ is a minimal $P(\gamma_k,p)$-exception and $G > G'$, $\gamma_{k}(G')$ is a $p$-nilpotent normal subgroup of $G$. Hence $\gamma_{k}(G') \leq Fit_p(G)$ is a $p$-group by Lemma \ref{Op'.word}. By Lemma \ref{delta.gamma.inclusion} applied to $G'$ we get that $G^{(k)} = (G')^{(k-1)} \leq \gamma_{k}(G')$. Thus $G^{(k)}$ is a $p$-group and $G$ is soluble.
\end{proof}

\begin{proof}[Proof of Theorem \ref{thm.gamma}]
If $\gamma_k(G)$ is $p$-nilpotent then $G$ satisfies property $P(\gamma_k,p)$ by Corollary \ref{thm.intro}.
Suppose that $G$ has property $P(\gamma_k,p)$. Aiming for a contradiction, suppose $G$ is minimal (with respect to the order) such that $\gamma_k(G)$ is not $p$-nilpotent. Then by Lemma \ref{gamma.G soluble} we get that $G$ is soluble and $G'$ has a normal Sylow $p$-subgroup $T$. Set $P=T \cap \gamma_k(G) \in \Syl_p(\gamma_k(G))$. Note that $G$ is a minimal $P(\gamma_k,p)$-exception, so by Lemma \ref{self.centered.sylow} we deduce that $C_{\gamma_k(G)}(P)  \leq P$. Since $P\norm \gamma_k(G)$, Lemma \ref{p.sbg.word} implies that every $\gamma_k$-value of $G$ has order divisible by $p$. Using Lemma \ref{Pavel.general} we conclude that every $\gamma_k$-value of $G$ has $p$-power order. Therefore $\gamma_k(G)\leq P$ is a $p$-group and so it is $p$-nilpotent, a contradiction. This completes the proof.
\end{proof}

\section{Proof of Theorem \ref{thm.delta.sol}}
In this section we consider the word $w=\delta_k$ and we prove Theorem \ref{thm.delta.sol}.

\begin{lemma}\label{p'.word.delta} Let $G$ be a finite group satisfying $P(\delta_k,p)$ and let $x\in G_{\delta_k}$ be a $\delta_k$-value of $p'$-order. Then for every element $g \in G$, the element $[g,x,x]$ has $p'$-order.
\end{lemma}

\begin{proof}
Let $g \in G$ be an element and let $x\in G_{\delta_k}$ be a $\delta_k$-value of $p'$-order. 
Note that 
\[ [g,x,x] = [x^{-g}, x]^{x} \]
is a $\delta_k$-value of $G$. Also,
\[ [g,x,x] \cdot x^{-1} = x^{-[g,x]}. \]
Now, $x^{-1}$ and $x^{-[g,x]}$ are $\delta_k$-values of $G$ of $p'$-order (their order is equal to the one of $x$). Since $G$ satisfies $P(\delta_k,p)$, we deduce that $[g,x,x]$ has $p'$-order.
\end{proof}

\begin{lemma}\label{p.sbg.word.delta} Let $G$ be a finite group satisfying $P(\delta_k,p)$. If $P$ is a $p$-subgroup of $G$ and $x\in N_G(P)$ is such that $x \in G_{\delta_k}$ has $p'$-order, then $[P,x] =1$.
\end{lemma}

\begin{proof}
Let $g\in P$ be an element. By Lemma \ref{p'.word.delta} the element $[g,x,x]$ has $p'$-order. 
On the other hand, since $x \in N_G(P)$ we have that $[g,x,x] \in P$. Therefore the only possibility is $[g,x,x]= 1$. This shows that $[P,x,x] = 1$. Hence by \cite[Theorem 5.3.6]{GOR} we get $[P,x] = [P,x,x] = 1$.
\end{proof}

\begin{proof}[Proof of Theorem \ref{thm.delta.sol}]
If $G^{(k)}$ is $p$-nilpotent then $G$ satisfies property $P(\delta_k,p)$ by Corollary \ref{thm.intro}.

Suppose $G$ is soluble and satisfies $P(\delta_k,p)$ but is  minimal such that $G^{(k)}$ is not $p$-nilpotent. Note that every proper subgroup and every proper quotient of $G$ is soluble, so $G$ is a minimal $P(\delta_k,p)$-exception.
By Lemma \ref{sbg.and.quotient.general} the group $G'$ satisfies $P(\delta_k,p)$. Since $G$ is soluble, we have $G > G'$ and so $G^{(k+1)} =\delta_k(G')$ is $p$-nilpotent. In particular $G^{(k+1)}$ is contained in $Fit_p(G)$, that is a $p$-group by  Lemma \ref{Op'.word}. Let $P\in Syl_p(G^{(k)})$. Then  $G^{(k+1)} \leq P$ and 
\[[P, G^{(k)}] \leq [G^{(k)}, G^{(k)}] = G^{(k+1)} \leq P.\] 
Thus $P$ is normal in $G^{(k)}$.
By the Schur-Zassenhaus theorem \cite[Theorem 6.2.1]{GOR} there exists a subgroup $H$ of $G^{(k)}$ of $p'$-order such that $G^{(k)}=PH$. Let $q_1, q_2, \dots, q_n$ be all prime numbers dividing the order of $G^{(k)}$ and distinct from $p$, with $q_i \neq q_j$ for every $i\neq j$. Let $Q_i$ be a Sylow $q_i$-subgroup of $G$ and set $\hat{Q_i} = Q_i \cap G^{(k)} \in \Syl_{q_i}(G^{(k)})$. Since $G$ is soluble, by Lemma \ref{delta.focal} the group $\hat{Q_i}$ is generated by $\delta_k$-values of $G$ lying in $Q_i$.
Note that $H = \langle \hat{Q_1}, \dots, \hat{Q_n} \rangle$, so $H$ is generated by $\delta_k$-values of $G$ of $p'$-order. Since $P\norm G^{(k)}$, we deduce that $[P, H] = 1$ by Lemma \ref{p.sbg.word.delta}. Thus $H$ is a normal $p$-complement of $G^{(k)}$ and $G^{(k)}$ is $p$-nilpotent, a contradiction. This completes the proof.
\end{proof}

\bibliographystyle{amsplain}
\bibliography{books}
 
\end{document}